\documentclass[leqno,12pt]{amsart} 
\usepackage{amssymb} 
\theoremstyle{plain} 
\newtheorem{theorem}{\indent\sc Theorem}[section] 
\newtheorem{lemma}[theorem]{\indent\sc Lemma}
\newtheorem{corollary}[theorem]{\indent\sc Corollary}
\newtheorem{proposition}[theorem]{\indent\sc Proposition}

\theoremstyle{definition} 
\newtheorem{definition}[theorem]{\indent\sc Definition}
\newtheorem{remark}[theorem]{\indent\sc Remark}
\newtheorem{example}[theorem]{\indent\sc Example}

\newtheorem*{theorem1}{\indent\sc Theorem A}
\newtheorem*{theorem2}{\indent\sc Theorem B}
\newtheorem*{case1}{\indent\sc $\underline{(\mathrm{I})\;n\geqslant4}$}
\newtheorem*{case2}{\indent\sc $\underline{(\mathrm{I}\hspace{-.1em}\mathrm{I})\;n=3}$}
\newtheorem*{case3}{\indent\sc $\underline{(\mathrm{I}\hspace{-.1em}\mathrm{I}\hspace{-.1em}\mathrm{I})\;n=2}$}

\begin{document}

\title[Rank 2 almost Fano bundles on $\mathbb{P}^n$]{On the classification of rank 2 almost Fano bundles on projective space} 

\author[K. Yasutake]{Kazunori Yasutake} 

\subjclass[2000]{ 
Primary 14J40; Secondary 14J10, 14J45, 14J60.
}

\keywords{ 
Projective manifold, anticanonical class, almost Fano variety, vector bundle.
}
\address{
Graduate School of Mathematical Sciences \endgraf
Kyushu University \endgraf
Fukuoka 819-0395 \endgraf
Japan
}
\email{k-yasutake@math.kyushu-u.ac.jp}

\maketitle

\begin{abstract}
An almost Fano bundle is a vector bundle on a smooth projective variety that its projectivization is an almost Fano variety. 
In this paper, we prove that almost Fano bundles exist only on almost Fano manifolds and study $\mathrm{rank}$ $2$ almost Fano bundles over projective spaces.
\end{abstract}

\section*{Introduction}
An almost Fano variety is a smooth projective variety whose anti-canonical line bundle is nef and big. 
This is a natural generalization of Fano varieties and often appears in the study of deformation of a Fano variety $($\cite{min}, \cite{sat}$)$. 
Almost Fano surfaces were completely classified by Demazure \cite{dem}. 
In \cite{lan}, Langer introduced the notion of weak Fano sheaves and classified almost  
Fano threefolds which is isomorphic to the projectivization of rank 2 sheaf 
$\mathcal{F}$ on $\mathbb{P}^2$ with $c_1(\mathcal{F})=-1$.
Recently Jahnke, Peternel and Radloff classified almost Fano threefolds with Picard number 2 
whose pluri-anti-canonical morphism is divisorial in \cite{j-p-r}. 
In \cite{take}, Takeuchi studied almost Fano threefolds with del Pezzo fibration structure 
whose pluri-anti-canonical morphism is small. 
In higher dimensional case, Jahnke and Peternell classified almost del Pezzo varieties, which are almost Fano $n$-folds with index $n-1$ i.e. its anti-canonical line bundle is divisible by $n-1$ in the Picard group. 

The aim of this paper is to study ruled almost Fano varieties $M$ of dimension $n\geqslant3$ over nonsingular variety $S$ i.e. there is a vector bundle $\mathcal{E}$ on $S$ such that $M$ is isomorphic to its projectivization $\mathbb{P}_S(\mathcal{E})$.  
\\

Now we recall the notion of almost Fano bundle, originally introduced by Langer as weak Fano bundle in \cite{lan}.

\begin{definition}
Let $\mathcal{E}$ be a vector bundle  on a smooth complex projective variety $M$. 
We say that  $\mathcal{E}$ is almost Fano if its projectivization $\mathbb{P}_M(\mathcal{E})$ is an almost Fano variety.
\end{definition}

Such bundles always exist on an almost Fano variety $M$. 
In fact, we notice that the trivial $\mathrm{rank}$ $r$ vector bundle is almost Fano since $\mathbb{P}_M(\mathcal{O}_M^{\oplus r})\cong M\times \mathbb{P}^{r-1}$ is also an almost Fano variety.  
In \cite[Theorem 1.6]{s-w3}, it is shown that Fano bundles are only on Fano manifolds.  
We consider the almost Fano case and obtain the following theorem.  

\begin{theorem1}
If $\mathcal{E}$ is an almost Fano bundle over a smooth complex projective variety $M$, then $M$ is an almost Fano variety.
\end{theorem1}

On projective spaces, $\mathrm{rank}$ 2 Fano bundles are completely classified in  \cite{a-p-w}, \cite{s-w3} and \cite{s-w2}. 
Using their methods, we study the classification of  $\mathrm{rank}$ $2$ almost Fano bundles on projective spaces and have the list mentioned below. 

\begin{theorem2}
Let $\mathcal{E}$ be a $\mathrm{rank}$ $2$ normalized $($i.e. $c_1(\mathcal{E})=0$ or $-1)$ almost Fano bundle on $\mathbb{P}^n$. 
Assume that $\mathcal{E}$ is not Fano. 
Then, $\mathcal{E}$ is isomorphic to one of the following :\\
$(1)$ $\mathcal{O}_{\mathbb{P}^n}(\lfloor \frac{n}{2}\rfloor) \oplus \mathcal{O}_{\mathbb{P}^n}(\lfloor -\frac{n}{2}\rfloor)$, where $\lfloor \frac{n}{2}\rfloor$ is the largest integer $\leqslant \frac{n}{2}$.\\
$(2)$ a stable bundle on $\mathbb{P}^3$ with $c_1=0$, $c_2=2$.\\
$(3)$ a stable bundle on $\mathbb{P}^3$ with $c_1=0$, $c_2=3$.\\
$(4)$ a vector bundle on $\mathbb{P}^2$ determined by the exact sequence : 
$0 \rightarrow \mathcal{O}_{\mathbb{P}^2} \rightarrow \mathcal{E} \rightarrow \mathcal{I}_p(-1) \rightarrow 0$, where $\mathcal{I}_p$ is the ideal sheaf of a point $p$.\\
$(5)$ a stable bundle on $\mathbb{P}^2$ with $c_1=-1$, $2 \leqslant c_2 \leqslant 5$.\\
$(6)$ a stable bundle on $\mathbb{P}^2$ with $c_1=0$, $4 \leqslant c_2 \leqslant 6$.
\end{theorem2}

Moreover, we show that all cases stated above really exist, except the case when $c_2=6$ in $(6)$. 
Note that these varieties are of index 1 or 2. 
On three dimensional projective space, the most difficult part is a construction of almost Fano bundles satisfying the condition in (3). 
To obtain this, we use Maruyama's theory of elementary transformation of vector bundles.
On projective plane, the case $c_1=-1$ was classified originally in \cite{lan} and later independently in \cite{j-p}. 
Therefore we treat the case $c_1=0$ i.e. $\mathbb{P}(\mathcal{E})$ is of index 1. 
In particular, we study almost Fano threefolds of index 1 whose pluri-anti-canonical morphism is small, having $\mathbb{P}^1$-bundle structure over $\mathbb{P}^2$. 

Ruled varieties play an important role in the classification theory of projective varieties.
So we may expect that our results also have applications.
\section*{Acknowledgements}
The author would like to express his gratitude to his supervisor Professor Eiichi Sato for many useful discussions and warm encouragements. He would also like to thank Professor Hiromichi Takagi for helpful advice, Professor Shunsuke Takagi and Professor Yu Kawakami for useful comments on an earlier draft of this paper.
\section*{Notation}
Throughout this paper $\mathcal{E}$ is a vector bundle on a smooth complex projective variety $M$ and $\xi_{\mathcal{E}}$ is the tautological line bundle on $X=\mathbb{P}_M(\mathcal{E})$. 
By $\pi$ we denote the projection $\pi: \mathbb{P}_M(\mathcal{E})\rightarrow M$ and by $H$ the pull-back of hyperplane if $M=\mathbb{P}^n$ $($i.e.
$\mathcal{O}_{\mathbb{P}_M(\mathcal{E})}(H)\cong\pi^*\mathcal{O}_{\mathbb{P}^n}(1)$). 
For a curve $C$ in $M$, we denote by $[C]$ the numerical equivalence class of $C$ in $M$. 

\section{Proof of Theorem A}
In this section we prove Theorem $A$. 
Before starting the proof, we prepare some facts.

\begin{definition}
Let $X$ be a normal projective variety and $\Delta$ an effective $\mathbb{Q}$-divisor on $X$. 
Let $\varphi : Y \rightarrow X$ be a log resolution of $(X,\Delta)$. 
We set 
\[K_Y=\varphi^*(K_X+\Delta)+\sum a_iE_i, \]
where $E_i$ is a prime divisor. 
The pair $(X,\Delta)$ is called kawamata log terminal (klt, for short) if $a_i>-1$ for all i.
\end{definition}

\begin{definition}
Let $X$ be a normal projective variety and $\Delta$ an effective $\mathbb{Q}$-divisor on $X$. 
We say that the pair $(X,\Delta)$ is a log Fano variety if $(X,\Delta)$ is klt and $-(K_X+\Delta)$ is an ample $\mathbb{Q}$-divisor.
\end{definition}

\begin{lemma}
If $X$ is an almost Fano manifold, there is an effective $\mathbb{Q}$-divisor $\Delta$ such that $(X,\Delta)$ is a log Fano variety.
\end{lemma}

\begin{proof}
For any ample divisor $A$, there are an integer m and an effective divisor $E$ such that $-mK_X=A+E$ by \cite[Lemma 2.60]{k-m}. 
Put $\Delta=\frac{1}{l}E$ for $l\gg0$, then $(X,\Delta)$ is klt from \cite{k-m}, corollary 2.35 and \[-lm(K_X+\Delta)=m(l-m)(-K_X)+mA\] is ample.
\end{proof}

Using this lemma, we get the following results by \cite{k-m} and \cite{z}.

\begin{theorem}\label{thm}
Let $X$ be an almost Fano manifold. 
Then, 
\\$(1)($Basepoint-free Theorem$)$  

Any nef divisor $D$ on $X$ is semiample $($i.e. bD is basepoint free for $b\gg0)$. 
\\$(2)($Cone Theorem$)$

There are finitely many rational curves $C_j$ on $X$ such that 
\[\overline{NE(X)}=\displaystyle \sum_{finite}\mathbb{R}_{\geqslant 0}[C_j].\] 
$(3)$ $X$ is rationally connected i.e. for any two points in $X$ there exists a rational curve which passes through them.
\end{theorem}

The next lemma is also needed.

\begin{lemma} $({\rm c.f.}$ \cite[Lemma 3.3]{wis}$)$. \label{wis}
Let $\pi :X=\mathbb{P}_M(\mathcal{E})\rightarrow M$ be the projectivization of a $\mathrm{rank}$ $r$ almost Fano bundle $\mathcal{E}$ and $C$  an extremal rational curve on 
$X$ not contracted by $\pi$. 
Then, we have $0\leqslant -K_X.C\leq -K_M.\pi (C)$.
\end{lemma}
\begin{proof}
Let $C$ be an extremal rational curve on $X$ not contracted by $\pi$ and $\varphi_C$ the corresponding elementary contraction map. 
Then $\varphi$ satisfies the assumption in  \cite[Lemma 3.3]{wis}. 
Hence we obtain the inequality in the lemma.
\end{proof}
\begin{proof}[Proof of Theorem A]
Put $X=\mathbb{P}_M(\mathcal{E})$. 
From Theorem \ref{thm}, we can find finitely many extremal rational curves $C_0$, $C_1$, $\cdots$, $C_{\rho}$ in X which generate the Kleiman-Mori cone $\overline{NE(X)}$. 
Let $C_0$ be contained in a fiber of the projection $\pi$. 
Then we see that $\overline{NE(M)}=\displaystyle \sum^{\rho }_{i=1}\mathbb{R}_{\geq 0}[\pi (C_i)]$. 
From Lemma \ref{wis}, it follows that \[-K_M.\pi (C_i)\geq -K_X.C_i\geq 0\] for $1\leqslant i\leqslant \rho$. 
Therefore $-K_M$ is nef.
Next we show the bigness of $-K_M$. 
Applying Theorem \ref{thm} (1) to $D:=\pi^*(-K_M)$, we know $D$ is semiample. Because $\pi$ is projective space bundle, $-K_M$ is also semiample. 
Let $\varphi=\varphi_{| -lK_M|}:M\rightarrow W$ be a morphism induced by $-lK_M$ for $l\gg0$. 
Suppose that $\mathrm{dim}M>\mathrm{dim}W$. 
Take the Stein factorization, we may assume a fiber of $\varphi$ is connected. 
We denote its general fiber by $F$. 
Then $F$ is smooth and we see that $-K_M|_F=-K_F$ holds. 
From this, 
\[-K_X|_{\pi^{-1}(F)}=(r\xi_{\mathcal{E}}-\pi^*(K_M+c_1(\mathcal{E})))|_{\pi^{-1}(F)}\]
\[\hspace{39.0mm}=r\xi_{\mathcal{E}|_F}-\pi^*(K_F+c_1(\mathcal{E}|_F))=-K_{\mathbb{P}_F(\mathcal{E}|_F)}.\] 
Therefore we may only consider $\varphi(M)$ is a point. 
In this case, Kodaira dimension $\kappa(M)$ of $M$ is equal to 0. 
On the other hand, $X$ is rationally connected due to Theorem \ref{thm} $(3)$. Since $\pi$ is surjective, $M$ is also rationally connected. 
Hence we have $\kappa(M)=-\infty$. 
This is a contradiction.
\end{proof}

\begin{remark}
$(1)$ This theorem is proved in \cite{c-j-r} if $\mathrm{dim}X=2$ and $\mathrm{rank}\mathcal{E}=2$.

$(2)$ Recently Fujino and Gongyo prove if $X$ is almost Fano and $f:X\rightarrow Y$ is a smooth morphism, then $Y$ should be almost Fano \cite{f-g}. 
\end{remark}
\section{Proof of Theorem B}
In this section, we study the structure of almost Fano bundles on projective space. 

First we consider almost Fano bundles which are decomposed into a direct sum of line bundles. 
In this case, we can characterize almost Fano bundles for any $\mathrm{rank}$.
\begin{proposition}\label{split}
Let $\mathcal{E} \cong \mathcal{O}\oplus \mathcal{O}(a_1)\oplus \mathcal{O}(a_2)\oplus \cdots \mathcal{O}(a_r)$ be a vector bundle on $\mathbb{P}^n$, where 
$0\leqslant a_1 \leqslant a_2 \leqslant \cdots \leqslant a_r$. 
Then, $\mathcal{E}$ is almost Fano if and only if $0\leqslant c_1(\mathcal{E})=$ $\displaystyle \sum^r_{i=1}a_i \leqslant n+1$.
Moreover $\mathcal{E}$ is not Fano if and only if $c_1(\mathcal{E})=n+1$.
\end{proposition}

\begin{proof}
Put $X=\mathbb{P}_{\mathbb{P}^n}(\mathcal{E})$. 
Then, we have $-K_X=r\xi_{\mathcal{E}}-(n+1-c_1(\mathcal{E}))H$. 
From the choice of $\mathcal{E}$, we can check naturally that $\mathcal{E}$ is Fano if and only if  $0\leqslant c_1(\mathcal{E}) \leqslant n$.
Next we will establish the latter part. It is easy to see that $-K_M$ is nef but not ample if and only if $c_1(\mathcal{E})=n+1$.  
Therefore it is sufficient to show that $-K_M$ is big if $c_1(\mathcal{E})=n+1$. 
In this case, $H^0(\xi_{\mathcal{E}}-H)\cong H^0(\mathcal{E}(-1))\neq 0$. 
By Kodaira's lemma, $-K_M=r\xi_{\mathcal{E}}=((r-1)\xi_{\mathcal{E}}+H)+(\xi_{\mathcal{E}}-H)$ is big. 
\end{proof}

From now on, we give a proof of Theorem B. 
The proof is divided into three parts, $($I$)\;n\geqslant4$, $($I$\hspace{-.1em}$I$)\;n=3$ and $($I$\hspace{-.1em}$I$\hspace{-.1em}$I$)\;n=2$.
\begin{case1}
\end{case1}
At first, we consider the case where $n\geqslant4$. 
The claim is as follows.
\begin{proposition}\label{prop1}
Let $\mathcal{E}$ be an almost Fano 2-bundle on $\mathbb{P}^n$, $n\geq 4$. 
Then $\mathcal{E}$ is a direct sum of two line bundles. 
\end{proposition}

These bundles are classified in Proposition \ref{split}. 
To show this, we need the next two lemmata.

\begin{lemma}\label{lemma1}
Let $\mathcal{E}$ be a normalized $\mathrm{rank}$ 2 almost Fano bundle on $\mathbb{P}^n$. If $n\geqslant 4$, then $\mathcal{E}(n)$ is generated by its global sections.
\end{lemma}

\begin{proof}
The proof is in the similar fashion as in \cite[Proposition 2.6]{a-p-w}. 
We give an outline of the proof in the case where $n$ is even and $c_1=-1$. 
Put $n=2k$ and $X=\mathbb{P}_{\mathbb{P}^n}(\mathcal{E})$,  then we have 
\[-K_X=2\xi_{\mathcal{E}}+(2k+2)H=2(\xi_{\mathcal{E}}+(k+1)H)\]
 is nef and big. 
 Therefore $\mathcal{E}(k+2)$ is ample vector bundle. 
 By Le Potier vanishing theorem, 
 \[H^i(\mathcal{E}(k+2+j)\otimes K_{\mathbb{P}^n})=H^i(\mathcal{E}(j-k+1))=0\]
 for any $i\geqslant 2$ and $j\geqslant 0$. 
 Especially  letting $j=3k-i-1$, we have $H^i(\mathcal{E}(n-i))=0$ for $i\geqslant2$. 
 Moreover 
 \[H^1(\mathbb{P}_{\mathbb{P}^n}(\mathcal{E}), 3(\xi_{\mathcal{E}}+(k+1)H)+(k-2)H+K_{\mathbb{P}_{\mathbb{P}^n}(\mathcal{E})})=H^1(\mathbb{P}^n, \mathcal{E}(n-1))=0\]
 from Kawamata-Vieweg vanishing theorem. 
Combining above, then we see that $H^i(\mathcal{E}(n-i))=0$ for $i\geqslant1$ namely $\mathcal{E}$ is $n-$regular. 
By means of Castelnuovo-Mumford lemma, $\mathcal{E}(n)$ is generated by its global sections. 
Other cases are proved in the same way.
\end{proof}

\begin{lemma}\label{lemma2}\cite{a-p-w} 
Let $\mathcal{E}$ be a globally generated 2-bundle on $\mathbb{P}^n$. Then we have

$(1)$ If $\mathcal{E}$ is not stable and $c_2(\mathcal{E})<(n-1)(c_1(\mathcal{E})-n+2)$, then $\mathcal{E}$ is split into a direct sum of two line bundles.

$(2)$ If $n\geqslant 6$ and $c_1(\mathcal{E})^2<4c_2(\mathcal{E})$, then we have $c_1(\mathcal{E})\geq 2n+3$.
\end{lemma}

\begin{proof}[Proof of Proposition \ref{prop1}]
Applying Lemma \ref{lemma2} to $\mathcal{E}(n)$, we can show immediately that $\mathcal{E}$ is split except for $n=4\;and\;5$ essentially in the same as in the proof of Proposition 3.1 and Proposition 5.1 in \cite{a-p-w}.
If $n=4$ $($resp. $n=5)$, then $\mathcal{E}(3)$ $($resp. $\mathcal{E}(4))$ is nef. From \cite[Proposition 9.2]{a-p-w} $($resp. \cite[Proposition 9.4]{a-p-w}$)$, $\mathcal{E}$ is split.
\end{proof}
\begin{case2}
\end{case2}
Next, we consider the case where $n=3$. 
To start with, we  demonstrate $\mathrm{rank}$ 2 almost Fano bundle on $\mathbb{P}^3$ is one of vector bundles below.
\begin{proposition}
Let $\mathcal{E}$ be a normalized almost Fano 2-bundle on $\mathbb{P}^3$.
Then $\mathcal{E}$ is isomorphic to a direct sum of two line bundles or one of the following :

$(1)$ stable vector bundle with $c_1=0$, $c_2=1$.

$(2)$ stable vector bundle with $c_1=0$, $c_2=2$.

$(3)$ stable vector bundle with $c_1=0$, $c_2=3$.
\end{proposition}

\begin{proof}
We shall discuss the two cases $c_1=0$ and $c_1=-1$ separately.

First we treat $c_1=-1$. 
Since $-K_X=2\xi_{\mathcal{E}}+5H$ is nef and big, we have that $\mathcal{E}(3)$ is ample. 
We can apply the argument in \cite[Theorem 2.2]{s-w3}, to this case and we can show that $\mathcal{E}$ is decomposed into a direct sum of two line bundles.

Next we treat $c_1=0$. 
In this case, $\mathcal{E}(2)$ is nef. 
If $H^0(\mathcal{E}(-2))\not =0$, then we can take a non-zero section $s\in H^0(\mathcal{E}(-2))$. 
If $Z:=\{s=0\}=\emptyset $, then $\mathcal{E}$ is decomposed into a direct sum of line bundles. 
If $Z\not=\emptyset $, then for a line $L$ meeting $Z$ in a finite number of points we would have 
\[\mathcal{E}(-2)|_L\cong \mathcal{O}_L(d)\oplus \mathcal{O}_L(-4-d), (d\geq 1)\]
which contradicts to the nefness of $\mathcal{E}(2)$.
If $H^0(\mathcal{E}(-2))=0$ and $H^0(\mathcal{E}(-1))\not =0$, then we can take a non-zero section $s\in H^0(\mathcal{E}(-1))$. 
If $Z:=\{s=0\}=\emptyset $, then $\mathcal{E}$ is decomposed into a direct sum of line bundles.
If $Z\not=\emptyset $, then $Z$ is a curve. 
Suppose that $\mathrm{deg}Z\geq 2$, we can take a line $L$ intersecting with $Z$ at least two points. 
Then 
\[\mathcal{E}(-1)|_L\cong \mathcal{O}_L(d)\oplus \mathcal{O}_L(-2-d), (d\geq 2)\] 
and contradict to the nefness of $\mathcal{E}(2)$.
If $\mathrm{deg}Z=1$, then $Z$ is a line. 
But, 
\[\mathrm{deg}K_Z=\mathrm{deg}(K_{\mathbb{P}^3}+c_1(\mathcal{E}(-1)))| _Z=-6.\]
This is a contradiction.
If $H^0(\mathcal{E}(-1))=0$ and $H^0(\mathcal{E})\not =0$, then we can take a non-zero section $s\in H^0(\mathcal{E})$. 
If $Z:=\{s=0\}=\emptyset $, then $\mathcal{E}$ is decomposed into a direct sum of two line bundles.
If $Z\not=\emptyset $, then $Z$ is a curve and $\mathrm{deg}Z=c_2\geq 1$. 
On the other hand, $\xi _{\mathcal{E}}(-K_X)^3=8-6c_2\geq 0$. Therefore $\mathrm{deg}Z=1$ i.e. $Z$ is a line. 
But, 
\[\mathrm{deg}K_Z=\mathrm{deg}(K_{\mathbb{P}^3}+c_1(\mathcal{E}))| _Z=-4.\] 
This is a contradiction.
Finally we assume $H^0(\mathcal{E})=0$ i.e. $\mathcal{E}$ is stable. 
In this case, $c_1^2<4c_2$ and $(-K_X)^4=128(4-c_2)>0$ hold. 
Hence $1\leq c_2\leq 3$. 
\end{proof}

It is shown \cite{s-w3} that all stable bundles satisfying $c_1=0$, $c_2=1$ are Fano. 
If $c_2=2$, then $\mathcal{E}$ is 2-regular by \cite{har}. 
Therefore $-K_X=2(\xi _{\mathcal{E}}+2H)$ is nef and big i.e. $\mathcal{E}$ is almost Fano. Such $\mathcal{E}$ is not Fano bundle \cite{s-w3}. 
The case $c_2=3$ is more complicated. 
First we show such an almost Fano bundle really exists. 

\begin{proposition}\label{0-3}
There is an almost Fano stable bundle on $\mathbb{P}^3$ with $c_1=0$, $c_2=3$.
\end{proposition}

To show this, we use the following result. 

\begin{theorem}{\rm\cite[Proposition 6]{mor} }\label{mor}
There is a nonsingular elliptic curve $C$ on a smooth quartic surface $S\subset \mathbb{P}^3$ and a very ample divisor $A$ on $S$ such that

$(1)\;\;Pic(S)\cong \mathbb{Z}[A]\oplus \mathbb{Z}[C].$ 

$(2)\;\;A^2=4,\; A.C=7,\; C^2=0.$

$(3)\;\;C\;is\;base\;point\;free.$

$(4)$  $S$ does not contain any rational curve.
\end{theorem}

\begin{proof}[Proof of Proposition \ref{0-3}]
Let $(S,C)$ be a pair in Theorem $\ref{mor}$. 
Using the theory of elementary transformation \cite{mar1} and \cite{mar2}, we can construct a $\mathrm{rank}$ $2$ regular vector bundle $\mathcal{F}$ on $\mathbb{P}^3$
where $c_1(\mathcal{E})=S,\;c_2(\mathcal{E})=C$ modulo numerical equivalence.
We will prove that $\mathcal{E}:=\mathcal{F}(-2)$ is the bundle what we want.
Since $\mathcal{F}$ has a global sections, we have a following exact sequence 
\[0 \rightarrow \mathcal{O}_{\mathbb{P}^3} \rightarrow \mathcal{F} \rightarrow \mathcal{I}_C(4) \rightarrow 0.\] 
Twist by $\mathcal{O}_{\mathbb{P}^3}(-2)$, we obtain 
\[0 \rightarrow \mathcal{O}_{\mathbb{P}^3}(-2) \rightarrow \mathcal{F}(-2) \rightarrow \mathcal{I}_C(2) \rightarrow 0.\] 
Because $C$ is not contained in any quadric surface, we see that $H^0(\mathcal{I}_C(2))=0$. 
Therefore $\mathcal{F}$ is stable since $H^0(\mathcal{F}(-2))=0$ and $c_1(\mathcal{F}(-2))=0$. 
Next we show $\mathcal{F}$ is nef. 
Note that $\mathcal{F}$ has 2 global sections which induce the generically surjective morphism $\varphi : \mathcal{O}^{\bigoplus 2} \rightarrow \mathcal{F}$ where $\varphi$ is isomorphic outside $S$ by the construction. 
Consequently $\mathcal{F}$ is nef over curves not contained in $S$. 
Over $S$, we get an exact sequence 
$0 \rightarrow \mathcal{O}_S(C) \rightarrow \mathcal{F}| _S \rightarrow \mathcal{O}_S(4A-C) \rightarrow 0$. From the choice of $C$, $\mathcal{O}_S(C)$ is nef. 
We have only to check the nefness of $\mathcal{O}_S(4A-C)$. 
Since $(aA+bC)(4A-C)=9a+28b$, we must prove that $9a+28b\geq 0$ if $aA+bC$ is effective.
But this is true since $(28A-9C)^2=-17<0$ and in view of Kleiman-Mori cone. 
Therefore $-K_X=2\xi _{\mathcal{F}}$ is nef and big.
Namely $\mathcal{F}$ is almost Fano. 
Hence $\mathcal{E}=\mathcal{F}(-2)$ is a stable 2-bundle with $c_1=0,\;c_2=3$ which is almost Fano. 
\end{proof}

Let $\mathcal{M}(0,3)$ be the moduli space of stable rank 2 vector bundles on $\mathbb{P}^2$ with $c_1=0$ and $c_2=3$.
From Theorem in \cite{e-s}, we see that $\mathcal{M}(0,3)$ has two irreducible components $\mathcal{M}_0(0,3)$ and  $\mathcal{M}_1(0,3)$ where 
$\mathcal{M}_{\alpha}(0,3)$ is the moduli space of vector bundles $\mathcal{E}$ satisfying the condition $\mathrm{dim}H^1(\mathcal{E}(-2))\equiv \alpha$ $(mod\;2)$. 
The dimension of each components are 21. 
Almost Fano bundles constructed in Proposition \ref{0-3} are contained in $\mathcal{M}_0(0,3)$. The author does not know whether $\mathcal{M}_1(0,3)$ contains almost Fano bundles or not.

Next we show that each components contain the member which is not almost Fano.
\begin{example}
From Proposition in \cite{rao}, we see that vector bundles in $\mathcal{M}_0(0,3)$ which have a maximal order jumping line is of dimension 20. 
Such a bundle $\mathcal{E}$ is decomposed into $\mathcal{O}_L(3)\oplus \mathcal{O}_L(-3)$ over some line $L$. 
These bundles cannot be almost Fano since $\mathcal{E}(2)$ is not nef.
\end{example}
\begin{example}
Let $Y$ be a disjoint union of a plane cubic and a nonsingular space elliptic curve in $\mathbb{P}^3$. 
By Serre construction, we can construct a $\mathrm{rank}$ 2 bundle $\mathcal{F}$ on $\mathbb{P}^3$ with $c_1=4$, $c_2=7$. 
Then, we can check $H^0(\mathcal{F}(-2))=0$ due to the exact sequence $0 \rightarrow \mathcal{O}_{\mathbb{P}^3} \rightarrow \mathcal{F} \rightarrow \mathcal{I}_Y(4) \rightarrow 0.$
Hence $\mathcal{F}$ is stable. 
Since every nonsingular space elliptic curve is a complete intersection of two quadrics, we have $H^0(\mathcal{I}_Y(3))=H^0(\mathcal{F}(-1))\not=0$.
From easy computation, $(\xi _{\mathcal{F}}-H).(-K_{\mathbb{P}(\mathcal{F})})^3=-1.$ 
Thus $\mathcal{E}:=\mathcal{F}(-2)$ is a stable vector bundle with $c_1=0$, $c_2=3$ which is not almost Fano. 
We can check 
\[\mathrm{dim}H^1(E(-2))=\mathrm{dim}H^1(\mathcal{I}_Y)=\mathrm{dim}H^0(\mathcal{O}_Y)-1=1\] 
using the exact sequence $0 \rightarrow \mathcal{I}_Y \rightarrow \mathcal{O}_{\mathbb{P}^3} \rightarrow \mathcal{O}_Y \rightarrow 0$. Hence $\mathcal{E}$ is contained in $\mathcal{M}_1(0,3)$.
\end{example}
\begin{case3}
\end{case3}
Finally, we consider the case where $n=2$. 
The case when $c_1=-1$ was completely classified in \cite[Theorem 3.2]{lan}. 
So we may only study bundles with $c_1=0$. 
The statement is as follows.
\begin{proposition}
Let $\mathcal{E}$ be a $\mathrm{rank}$ 2 almost Fano bundle on $\mathbb{P}^2$ with $c_1=0$. 
Then, $\mathcal{E}$ is isomorphic to one of the folowing

$(1)$ $\mathcal{O}_{\mathbb{P}^2}(1)\bigoplus \mathcal{O}_{\mathbb{P}^2}(-1)$,

$(2)$ $\mathcal{O}_{\mathbb{P}^2}\bigoplus \mathcal{O}_{\mathbb{P}^2}$,

$(3)$ $\mathcal{E}$ is determined by $0 \rightarrow  \mathcal{O}_{\mathbb{P}^2} \rightarrow \mathcal{E} \rightarrow \mathcal{I}_p \rightarrow 0$, where $\mathcal{I}_p$ is the ideal sheaf of a point,

$(4)$ stable vector bundle with $2 \leqslant c_2 \leqslant 6$.
\end{proposition}

\begin{proof}
In this case, $\mathcal{E}(2)$ is ample. 
If $H^0(\mathcal{E}(-1))\neq0$, we take a non-zero section $s \in H^0(\mathcal{E}(-1))$. 
If $Z:=\{s=0\}= \emptyset$, then $\mathcal{E}$ is decomposed into a direct sum of line bundles. 
If $Z\not=\emptyset$, then for a line $L$ meeting $Z$ in a finite number of points we would have
\[\mathcal{E}(-1)|_L\cong \mathcal{O}_L(d)\oplus \mathcal{O}_L(-2-d),\;d\geqslant1.\]
This contradict to the ampleness of $\mathcal{E}(2)$.

If $H^0(\mathcal{E}(-1))=0$ and $H^0(\mathcal{E})\not=0$, take a non-zero section $s \in H^0(\mathcal{E})$. 
If $Z:=\{s=0\}=\emptyset$, then $\mathcal{E}$ is decomposed into a direct sum of line bundles. 
If $Z\not= \emptyset$ and $\mathrm{deg}Z\geq 2$, then for a line $L$ intersecting with $Z$ at least two points we would have 
\[\mathcal{E}| _L\cong \mathcal{O}_L(d)\oplus \mathcal{O}_L(-d),\;d\geq 2.\]
This is a contradiction. 

If $\mathrm{deg}Z=1$, $\mathcal{E}$ has an exact sequence $0 \rightarrow \mathcal{O}_{\mathbb{P}^2} \rightarrow \mathcal{E} \rightarrow \mathcal{I}_p \rightarrow 0$, where $\mathcal{I}_p$ is the ideal sheaf of a point $p$. 
In this case $\mathcal{E}$ is Fano bundle by \cite[Proposition 2.3]{s-w2}.
Finally we consider the case $H^0(\mathcal{E})=0$ i.e. $\mathcal{E}$ is stable. 
Then $2\leq c_2\leq 6$ since $(-K_X)^3=54-8c_2>0$.
\end{proof}

We have some comments of Fano bundles with $c_1=0$. 
If $c_2=2$, all stable bundles are Fano bundle from \cite{s-w2}. 
In the situation $c_2=3$, there is a stable Fano bundles by \cite{s-w2}. 
Moreover, we have the following result.

\begin{proposition}\label{0-3-2}
If $\mathcal{E}$ is a stable almost Fano bundle on a projective plane with $c_1=0$, $c_2=3$.
Then $\mathcal{E}$ is Fano bundle.
\end{proposition}

\begin{proof}
Let $\mathcal{E}$ be a stable almost Fano bundle on a projective plane with $c_1=0$, $c_2=3$. 
Using Riemann-Roch theorem, we have $\mathrm{dim}H^0(\mathcal{E}(1))>0$. 
Therefore we get an exact sequence $0\rightarrow \mathcal{O}_{\mathbb{P}^2}\rightarrow \mathcal{E}(1) \rightarrow \mathcal{I}_Z(2) \rightarrow 0$ where $Z$ is 4 points in $\mathbb{P}^2$ and $\mathcal{I}_Z$ is the ideal sheaf of $Z$.
If $\mathcal{E}$ is not Fano, then the linear system $|\xi_{\mathcal{E}}+H|$ has one dimensional base locus $B$ by \cite{s-w2}, Claim 2.7. 
By virtue of Claim 2.10 and 2.11 in \cite{s-w2},  we have $H.B\leqslant2$ and $(\xi_{\mathcal{E}}+H).B\leqslant-1$. 
Since $0\leqslant-K_{\mathbb{P}_{\mathbb{P}^2}(\mathcal{E})}.B=2(\xi_{\mathcal{E}}+H).B+H.B\leqslant0$, we obtain $H.B=2$ and $(\xi_{\mathcal{E}}+H).B=-1$. 
If $\pi(B)$ is a line $L$, we have $\mathcal{E}(1)|_L\cong \mathcal{O}(d)\oplus \mathcal{O}(2-d)$, $d\geqslant3$. 
This contradicts the ampleness of $\mathcal{E}(2)$. 
If $\pi(B)$ is a two line, take a irreducible component $L$. In this case $\mathcal{E}(1)$ is not nef over  $\pi(B)$, so over $L$. 
Therefore we have $\mathcal{E}(1)|_L\cong \mathcal{O}(d)\oplus \mathcal{O}(2-d)$, $d\geqslant3$. 
This contradicts the ampleness of $\mathcal{E}(2)$. 
Finally we consider the case where $\pi(B)$ is nonsingular conic $C$. 
Since $(\xi_{\mathcal{E}}+H).B=-1$, we obtain the splitting $\mathcal{E}(1)|_C\cong \mathcal{O}_C(d)\oplus \mathcal{O}_C(4-d), d\geqslant5$. 
This is impossible because $Z$ is only 4 points.
\end{proof}

\begin{corollary}
Let  $\mathcal{E}$ be a stable vector bundle on a projective plane with $c_1=0$, $c_2=3$. If $\mathcal{S}^2(\mathcal{E})(3)$ is nef, then $\mathcal{E}(1)$ is generated by global sections.
\end{corollary}
\begin{proof}
If $\mathcal{S}^2(\mathcal{E})(3)$ is nef, then $\mathcal{E}$ is almost Fano. 
From Proposition \ref{0-3-2}, $\mathcal{E}$ is Fano bundle. 
By means of Proposition 2.6 in \cite{s-w2}, $\mathcal{E}(1)$ is generated by global sections.
\end{proof}

When $c_2=4$, no stable 2-bundle is Fano \cite{s-w2}. 
We can constract almost Fano 2 bundle with $c_1=0, c_2=4$ as follows. 

\begin{example}
Let $Y$ be 5 points in general position and $C$ is a smooth conic containing $Y$. 
Then the pair $(C,Y)$ yields us a rank 2 regular vector bundle $\mathcal{F}$ with $c_1=C$, $c_2=Y$ by virtue of an elementary transform by \cite{mar1} and \cite{mar2}. 
We have a following exact sequence 
\[0 \rightarrow \mathcal{O}_{\mathbb{P}^2} \rightarrow \mathcal{F} \rightarrow \mathcal{I}_Y(2) \rightarrow 0\]
where $\mathcal{I}_Y$ is the ideal sheaf of $Y$. 
Twist by $\mathcal{O}_{\mathbb{P}^2}(-1)$, we get 
\[0 \rightarrow \mathcal{O}_{\mathbb{P}^2}(-1) \rightarrow \mathcal{F}(-1) \rightarrow \mathcal{I}_Y(1) \rightarrow 0.\] 
Because there is no line containing $Y$, we have $H^0(\mathcal{I}_Y(1))=0$. 
Therefore $\mathcal{F}$ is stable since $H^0(\mathcal{F}(-1))=0$ and $c_1(\mathcal{F}(-1))=0$. 
We check $-K_{\mathbb{P}_{\mathbb{P}^2}(\mathcal{F})}=2\xi_{\mathcal{F}}+H$ is nef. 
First we remark that $\mathcal{F}$ has 2 global sections which induce the generically surjective morphism $\varphi : \mathcal{O}^{\bigoplus 2} \rightarrow \mathcal{F}$ where $\varphi$ is isomorphic outside $C$ by the construction. 
Hence we notice that $2\xi_{\mathcal{F}}+H$ is nef outside $\pi^{-1}(C)$. 
On $C$, we have that $\mathcal{F}| _C\cong \mathcal{O}_C(5)\oplus \mathcal{O}_C(-1)$ from the theory of elementary transformation.  
From this fact, we can check that $(2\xi_{\mathcal{F}}+H).D\geqslant0$ for every curves $D$ contained in Hirzebruch surface  $\mathbb{P}_C(\mathcal{F|_C})$. 
The equality holds only for the minimal section associated with the quotient line bundle $\mathcal{F}|_C \rightarrow \mathcal{O}_C(-1) \rightarrow 0.$
Therefore $-K_{\mathbb{P}_{\mathbb{P}^2}(\mathcal{F})}$ is nef and big. 
Hence $\mathcal{E}:=\mathcal{F}(-1)$ is a stable almost Fano bundle with $c_2=4$. 
\end{example}

There exists an almost Fano stable bundle with $c_1=0$, $c_2=5$ from Theorem 0.19(C) in \cite{taka}. 
Finally we construct stable vector bundles with $c_1=0$, $3\leqslant c_2\leqslant6$ which are not almost Fano.

\begin{example}
Let $Y_k=\{p_0, p_1, \cdots, p_k\}$ be the $k+1$ points $(4\leqslant k\leqslant7)$ in $\mathbb{P}^2$. 
We assume that $p_0$, $p_1$, $p_2$ are lying in a line $L$ and other points are not on $L$.
By Serre construction, we have rank 2 vector bundles $\mathcal{E}_k$ on $\mathbb{P}^2$ with $c_1=2$, $c_2=k+1$. 
$\mathcal{E}_k$ has an exact sequence $0 \rightarrow \mathcal{O}_{\mathbb{P}^2} \rightarrow \mathcal{E}_k \rightarrow \mathcal{I}_{Y_k}(2) \rightarrow 0.$
Becuase there is no line containing $Y_k$, we see $\mathrm{dim}H^0(\mathcal{E}_k(-1))=\mathrm{dim}H^0(\mathcal{I}_{Y_k})$=0. 
Combining with $c_1(\mathcal{E}_k(-1))=0$, $\mathcal{E}_k$ is stable bundle. 
Restricting each bundles into $L$, we get $\mathcal{E}_k|_L\cong \mathcal{O}_L(3)\oplus \mathcal{O}_L(-1)$. 
Since $\mathcal{E}_k(1)$ is not ample, $\mathcal{E}_k$ is not almost Fano.
\end{example}

\end{document}